\documentclass[11pt]{amsart}

\usepackage{amsmath, amsthm,afterpage}

\usepackage{amssymb,amscd,amstext,units,mathtools,color,xcolor,enumerate,bm,pifont,comment}
\usepackage[mathscr]{eucal}
\usepackage[T1]{fontenc}
\usepackage[utf8]{inputenc}
\usepackage[all]{xy}
\usepackage[hidelinks]{hyperref}

\usepackage{geometry}
\theoremstyle{plain}
\newtheorem*{cor*}{Corollary}
\newtheorem*{prop*}{Proposition}
\newtheorem*{thm*}{Theorem}
\newtheorem*{notation*}{Notation}
\newtheorem{thm}{Theorem}[section]
\newtheorem{cor}[thm]{Corollary}
\newtheorem{prop}[thm]{Proposition}

\newtheorem{fact}[thm]{Fact}
\newtheorem{claim}[thm]{Claim}

\theoremstyle{definition}
\newtheorem{dfn}[thm]{Definition}
\newtheorem{rem}[thm]{Remark}

\numberwithin{equation}{subsection}

\newcommand{\forkindep}[1][]{%
  \mathrel{
    \mathop{
      \vcenter{
        \hbox{\oalign{\noalign{\kern-.3ex}\hfil$\vert$\hfil\cr
              \noalign{\kern-.7ex}
              $\smile$\cr\noalign{\kern-.3ex}}}
      }
    }\displaylimits_{#1}
  }
}  %%% $A\forkindep[C]B$

\newtheorem{innercustomgeneric}{\customgenericname}
\providecommand{\customgenericname}{}
\newcommand{\newcustomtheorem}[2]{%
  \newenvironment{#1}[1]
  {%
   \renewcommand\customgenericname{#2}%
   \renewcommand\theinnercustomgeneric{##1}%
   \innercustomgeneric
  }
  {\endinnercustomgeneric}
}

\newcustomtheorem{customthm}{Theorem}

\title{Definably compact groups definable in real closed fields.{II}}

\author{Eliana Barriga}
\address{ Eliana Barriga\\ Universidad de los Andes, Colombia\\
University of Haifa, Israel}
\email{el.barriga44@uniandes.edu.co}


\begin{document}

\begin{abstract}

We continue the analysis of definably compact groups definable in a real closed field $\mathcal{R}$. In \cite{BADefComp-I}, we proved that for every definably compact definably connected semialgebraic group $G$ over $\mathcal{R}$ there are a connected $R$-algebraic group $H$, a definable injective map $\phi$ from a generic definable neighborhood of the identity of $G$ into the group $H\left(R\right)$ of $R$-points of $H$ such that $\phi$ acts as a group homomorphism inside its domain. The above result and our study of locally definable covering homomorphisms for locally definable groups combine to prove that if such group $G$
is in addition abelian, then its o-minimal universal covering group
$\widetilde{G}$ is definably isomorphic, as a locally definable group,
to a connected open locally definable subgroup of the o-minimal
universal covering group $\widetilde{H\left(R\right)^{0}}$ of the
group $H\left(R\right)^{0}$ for some connected $R$-algebraic group
$H$.
\end{abstract}

\maketitle

\section{Introduction}

This is the second paper of two papers studying definably compact groups definable in real closed fields.

This paper offers a description of the semialgebraically connected
semialgebraic groups over a sufficiently saturated real closed field $R$ through the study
of their o-minimal universal covering groups (see Def. \ref{D:DefUnivCov})
and of their relation with the $R$-points of some connected $R$-algebraic
group.

We establish a connection between the o-minimal universal covering
groups of an abelian definably compact definably connected group definable
in $\mathcal{R}$ and of the semialgebraically connected component
$H\left(R\right)^{0}$ of the group $H\left(R\right)$ of $R$-points
of some connected $R$-algebraic group $H$. More precisely we show the following.

\begin{customthm}{\ref{T:Goal1}}
\textit{Let $R$ be a sufficiently saturated real closed field. Then, the o-minimal universal covering group of an abelian definably compact
definably connected group definable in $\mathcal{R}$ is an open
locally definable subgroup of the o-minimal universal covering group
of the semialgebraically connected component $H\left(R\right)^{0}$
of the group $H\left(R\right)$ of $R$-points of some connected $R$-algebraic
group $H$.}
\end{customthm}

To prove this result we apply Theorem 5.1 of the first paper (\cite{BADefComp-I}) and some results on locally definable covering homomorphisms of locally definable groups proved in this paper. Theorem 5.1 in \cite{BADefComp-I} asserts that for every definably compact definably connected semialgebraic group $G$ over $\mathcal{R}$ there are a connected $R$-algebraic group $H$ and a definable local homomorphism $\phi$ from a generic definable neighborhood of the identity of $G$ into the group $H\left(R\right)$ of $R$-points of $H$, where by a local homomorphism we mean the following.

\begin{dfn}\label{D:homomoloc}
Let $G_{1}$ and $G_{2}$ be two topological groups, $X\subseteq G_{1}$ a neighborhood of the identity of $G_{1}$, and $\phi:X\rightarrow G_{2}$ a map. $\phi$ is called a \textit{local homomorphism} if $x,y,xy\in X$ implies $\phi\left(xy\right)=\phi \left(x\right)\phi\left(y\right)$. We say that an injective map $\phi:X\subseteq G_{1}\rightarrow G_{2}$ is a \textit{local homomorphism in both directions} if $\phi:X\rightarrow G_{2}$ and $\phi^{-1}:\phi\left(X\right)\rightarrow X$ are local homomorphisms.
\end{dfn}

The strategy to prove Theorem \ref{T:Goal1} is to use the $R$-algebraic group $H$ and the local homomorphism given by \cite[Theorem 5.1]{BADefComp-I} to define a locally definable map $\overline{\theta}$ from some open locally definable subgroup of the o-minimal universal covering group of $H\left(R\right)$ to $G$ such that $\overline{\theta}$ works as the o-minimal universal covering group of $G$.

This research is part of my PhD thesis at the Universidad de los Andes, Colombia and University of Haifa, Israel.

\subsection{The structure of the paper}

In Section \ref{S:2} we regard the locally definable spaces, ld-spaces, and their ld-covering maps as are defined in \cite{BarOter}. The $\bigvee$-definable groups in $\mathcal{M}$ are examples of such
spaces. We show that any ld-covering map between ld-spaces is closed for definable subspaces (Prop. \ref{P:covisclosed}). The o-minimal universal covering homomorphism of a connected locally definable group
is introduced and studied in Section \ref{S:3}. Sections \ref{S:4} and \ref{S:5} investigate the connection between abelian definably generated groups, existence of generic definable sets, convex sets,
and covers of definable groups. We prove in Proposition \ref{P:ImpProp} the existence of a convex set inside a definable generic subset of an abelian $\bigvee$-definable group $\mathcal{U}$ with $\mathcal{U}^{00}$.
This is a crucial fact in the construction of a well defined covering map in Theorem \ref{T:Z_t-f}.

In the last two sections we develop some results on local homomorphisms and their extensions to locally definable homomorphisms. Finally, in Section \ref{S:7}, by means of Theorems \ref{T:UnivCovAbLD} and  \cite[Theorem 5.1]{BADefComp-I}, we prove the main result of this paper: Theorem \ref{T:Goal1}.

\begin{notation*}\label{N:}
Our notation and any undefined term that we use from model theory, topology, or algebraic geometry are generally standard. For a group $G$ whose group operation is written multiplicatively, we use the following notation $\prod_{n}X=\underset{n\textrm{-times}}{\underbrace{X\cdot\ldots\cdot X}}$,
and $X^{n}=\left\{ x^{n}:x\in X\right\} $ for any $n\in\mathbb{N}$.
\end{notation*}

\section{Ld-spaces and ld-covering maps}\label{S:2}

\textit{From now until the end of this paper, unless stated otherwise, we work over a sufficiently saturated o-minimal expansion of a real closed field $R$, where by a \textit{sufficiently saturated structure} we mean a $\kappa$-saturated structure for some sufficiently large cardinal $\kappa$.}

In \cite{BarOter} Baro and Otero introduced the locally definable category, which extends the locally semialgebraic one introduced by Delfs and Knebusch in \cite{DK} and is more flexible than the $\bigvee$-definable group category.  $\bigvee$-definable groups are examples of locally definable spaces and their locally definable covering homomorphisms are locally definable covering maps of locally definable spaces. Following, we will introduce some definitions of the locally definable category from \cite{BarOter}, and then we will prove some results on locally definable covering maps that will be applied later in the study of the locally definable covering homomorphisms of locally definable groups.

\subsection{Ld-spaces and ld-maps}

\begin{dfn}\label{D:}

Let $M$ be a set. A \textit{locally definable space} is a triple
$\left(M,\left(M_{i},\phi_{i}\right)_{i\in I}\right)$ where

\begin{enumerate}[(i)]

\item $M_{i}\subseteq M$, $M=\bigcup_{i\in I}M_{i}$, and $\phi_{i}:M_{i}\rightarrow Z_{i}$
is a bijection between $M_{i}$ and a definable set $Z_{i}\subseteq R^{n\left(i\right)}$
for every $i\in I$,

\item $\phi_{i}\left(M_{i}\cap M_{j}\right)$ is a definable relative
open subset of $Z_{i}$ and the transition maps $\phi_{ij}=\phi_{j}\circ\phi_{i}^{-1}:\phi_{i}\left(M_{i}\cap M_{j}\right)\rightarrow M_{i}\cap M_{j}\rightarrow\phi_{j}\left(M_{i}\cap M_{j}\right)$
are definable for every $i,j\in I$.

\end{enumerate}

The \textit{dimension} of $M$ is $\dim\left(M\right)\coloneqq\sup\left\{ \dim\left(Z_{i}\right):i\in I\right\}$.
If $Z_{i}$ and $\phi_{ij}$ are definable over $A\subseteq R$ for
all $i,j\in I$, we say that $M$ is a locally definable space over
$A$.

\end{dfn}

Note that every definable space (\cite[Chapter 10]{LVD}) is a locally
definable space with $\left|I\right|<\aleph_{0}$.

Every locally definable space $\left(M,\left(M_{i},\phi_{i}\right)_{i\in I}\right)$
has a unique topology on $M$ such that each $M_{i}$ is open and $\phi_{i}$
is a homeomorphism for all $i\in I$; more precisely, $\mathcal{O}\subseteq M$
is open if and only if $\phi_{i}\left(\mathcal{O}\cap M_{i}\right)$
is relatively open in $Z_{i}$ for every $i\in I$. Throughout this
subsection any topological property of locally definable spaces refers
to this topology.

\begin{dfn}\label{D:}

Let $\left(M,\left(M_{i},\phi_{i}\right)_{i\in I}\right)$ be an ld-space.

\begin{enumerate}[(i)]

\item An \textit{ld-space} is a Hausdorff locally definable space.

\item A subset $X\subseteq M$ is called a \textit{definable subspace}
of $M$ if there is a finite $J\subseteq I$ such that $X\subseteq\bigcup_{j\in J}M_{j}$
and $\phi_{j}\left(X\cap M_{j}\right)$ is definable for all $j\in J$.

\item A subset $Y\subseteq M$ is called an \textit{compatible subspace}
of $M$ if $\phi_{i}\left(Y\cap M_{i}\right)$ is definable for every
$i\in I$, or equivalently, $Y\cap X$ is a definable subspace of
$M$ for every definable subspace $X$ of $M$.

\end{enumerate}

\end{dfn}

By Theorem 3.9 of \cite{BarOter}, every $\bigvee$-definable group
$\mathcal{U}$ with its $\tau$-topology (see \cite[Lemma 7.5]{HPePiNIP})
is an ld-space of finite dimension, and any definable subset of $\mathcal{U}$ is a definable
subspace of $\mathcal{U}$.

We recall that any compatible subspace $Y$ of an ld-space $M$ inherits
a natural structure of ld-space \cite[Remark 2.3]{BarOter} given
by $\left(Y,Y_{i}=Y\cap M_{i},\phi_{i}\mid_{Y_{i}}\right)$. And if
$Y$ is a definable subspace then it inherits the structure of a definable
space. Note that the only compatible subspaces of a definable space
are the definable ones.

Now, we will introduce the maps between ld-spaces as in \cite{BarOter}.
For this we note that given two ld-spaces $\left(M,\left(M_{i},\phi_{i}\right)_{i\in I}\right)$
and $\left(N,\left(N_{j},\psi_{j}\right)_{j\in J}\right)$ we can
endow $M\times N$ with the structure $\left(M\times N,\left(M_{i}\times N_{j},\left(\phi_{i},\psi_{j}\right)\right)_{i\in I}\right)$
that makes it into an ld-space, and as it is defined in \cite{LVD},
a map $f:M\rightarrow N$ between definable spaces $M,N$ is a \textit{definable map}
if its graph is a definable subspace of $M\times N$ .

\begin{dfn}\label{D:}

A map $\theta:M\rightarrow N$ between ld-spaces (locally definable spaces) $\left(M,\left(M_{i},\phi_{i}\right)_{i\in I}\right)$
and $\left(N,\left(N_{j},\psi_{j}\right)_{j\in J}\right)$ is called
an \textit{ld-map} (\textit{locally definable map}) if $\theta\left(M_{i}\right)$ is a definable
subspace of $N$ and $\theta\mid_{M_{i}}:M_{i}\rightarrow \theta\left(M_{i}\right)$
is definable for every $i\in I$.

\end{dfn}

\subsection{Some topological notions in ld-spaces}

\begin{dfn}\label{D:ConnLdSpa}

Let $M$ be an ld-space.

\begin{enumerate}[(i)]

\item $M$ is \textit{connected} if $M$ has no compatible nonempty
proper clopen subspace.

\item An \textit{ld-path} in $M$ is a continuous ld-map $\alpha:\left[0,1\right]\rightarrow M$.

\item\inputencoding{latin1}{ }\inputencoding{latin9}$M$ is \textit{path connected}
if for every $x_{1},x_{2}\in M$ there is an ld-path $\alpha:\left[0,1\right]\rightarrow M$
such that $\alpha\left(0\right)=x_{1}$ and $\alpha\left(1\right)=x_{2}$.

\item The \textit{path connected component} of a point $x\in M$
is the set of all $y\in M$ such that there is an ld-path from $x$
to $y$.

\end{enumerate}

\end{dfn}

By Remarks 4.1 and 4.3, and Fact 4.2 of \cite{BarOter}, (i) an ld-space $M$ is connected if and only if $M$ is path connected if and only if every ld-map from $M$ to a discrete ld-space is constant, and (ii) every path connected component of an ld-space is a clopen compatible subspace.

\begin{claim}\label{C:UnionConn}
Let $M=\bigcup_{i\in I}X_{i}$ be an ld-space such that $\left\{ X_{i}:i\in I\right\} $
is a collection of connected compatible subspaces of $M$ and $\bigcap_{i\in I}X_{i}\neq\emptyset$.
Then $M$ is connected.
\end{claim}
\begin{proof}
Let $Y\subseteq M$ be a compatible nonempty clopen in $M$. Since
$Y\neq\emptyset$, there is $k\in I$ such that $Y\cap X_{k}\neq\emptyset$.
Since $X_{k}$ and $Y$ are compatible in $M$, so is $Y\cap X_{k}$,
and in particular $Y\cap X_{k}\subseteq X_{k}$ is a clopen compatible
set in $X_{k}$. By the connectedness of $X_{k}$, $Y\cap X_{k}=X_{k}$.

As $\bigcap_{i\in I}X_{i}\neq\emptyset$, $X_{i}\cap X_{k}\neq\emptyset$
for every $i\in I$, then $X_{i}\cap Y\neq\emptyset$, and as above
we conclude that $Y\cap X_{i}=X_{i}$ for every $i\in I$. Therefore,
$M=\bigcup_{i\in I}X_{i}=\bigcup_{i\in I}Y\cap X_{i}=Y$. Then $M$
has no clopen proper nonempty compatible subset.

\end{proof}

\begin{cor}\label{C:CartProdConn}
Let $M,N$ be two connected ld-spaces. Then the product ld-space $M\times N$
is connected.
\end{cor}
\begin{proof}
Fix $y\in N$. For $x\in M$, let $T_{x}=\left(\left\{ x\right\} \times N\right)\cup\left(M\times\left\{ y\right\} \right)$.
Since $\left(x,y\right)\in\left(\left\{ x\right\} \times N\right)\cap\left(M\times\left\{ y\right\} \right)$,
Claim \ref{C:UnionConn} implies that $T_{x}$ is connected. Finally,
as $\bigcap_{x\in M}T_{x}=M\times\left\{ y\right\} $, again Claim
\ref{C:UnionConn} implies that $\bigcup_{x\in M}T_{x}=M\times N$
is connected.
\end{proof}

\begin{prop}\label{P:DefGenIsConn}
Let $\mathcal{U}$ be a locally definable group and $X\subseteq\mathcal{U}$
a connected definable set such that the identity element
$e_{\mathcal{U}}\in X$. Then the definable generated group $\left\langle X\right\rangle $
is a connected locally definable group.
\end{prop}
\begin{proof}
By Corollary \ref{C:CartProdConn}, $\underbrace{X\times\cdots\times X}_{i\textrm{-times}}$
is a connected definable space for every $i\in\mathbb{N}$. Since the ld-map
\begin{align*}
p_{i}:\left\langle X\right\rangle \times\cdots\times\left\langle X\right\rangle  & \rightarrow  \left\langle X\right\rangle \\
\left(x_{1},\ldots,x_{i}\right) & \mapsto  \prod_{1\leq j\leq i}x_{j}
\end{align*} is continuous (with respect to their topologies of locally definable
groups) and the image of a connected ld-space by a continuous ld-map
is connected, then $p_{i}\left(\underbrace{X\times\cdots\times X}_{i\textrm{-times}}\right)=\prod_{i}X$
is connected.

Finally, as $\bigcap_{i\in\mathbb{N}}\prod_{i}\left(X\cup X^{-1}\right)\supseteq\left(X\cup X^{-1}\right)\neq\emptyset$,
then $\bigcup_{i\in\mathbb{N}}\prod_{i}\left(X\cup X^{-1}\right)=\left\langle X\right\rangle $
is connected by Claim \ref{C:UnionConn}.
\end{proof}

\begin{dfn}\label{D:FundGp}

Let $M$ be an ld-space and $x_{0}\in M$. Let $\alpha,\gamma:\left[0,1\right]\rightarrow M$
be two ld-paths. A continuous ld-map $H\left(t,s\right):\left[0,1\right]\times\left[0,1\right]\rightarrow M$
is a \textit{homotopy} between $\alpha$ and $\gamma$ if $\alpha=H\left(\cdot,0\right)$
and $\gamma=H\left(\cdot,1\right)$. In this case, $\alpha$ and $\gamma$
are called \textit{homotopic}, denoted $\alpha\sim\gamma$.

Let $\mathbb{L}\left(M,x_{0}\right)$ be the set of all ld-paths that
start and end at the element $x_{0}\in M$. Note that being homotopic
$\sim$ is an equivalence relation on $\mathbb{L}\left(M,x_{0}\right)$.
We define \textit{the o-minimal fundamental group} $\pi_{1}\left(M,x_{0}\right)\coloneqq\mathbb{L}\left(M,x_{0}\right)/\sim$.
Observe that $\pi_{1}\left(M,x_{0}\right)$ is a group with the operation
given by the class of the concatenation of its representatives; i.e.,
$\left[\alpha\right]\cdot\left[\gamma\right]=\left[\alpha\cdot\gamma\right]$.
In case $M$ is a connected locally definable group, $\pi_{1}\left(M,x_{0}\right)$
is an abelian group (\cite[Prop. 4.1]{ED05}).

$M$ is called \textit{simply connected} if $M$ is path connected and $\pi_{1}\left(M,x_{0}\right)$
is the trivial group.

\end{dfn}

\subsection{Covering maps for ld-spaces}

The next definition of covering map for ld-spaces is taken from \cite{BarOter}.

\begin{dfn}\label{D:ld-coveringMap}

Let $\left(M,\left(M_{i},\phi_{i}\right)_{i\in I}\right)$ and $\left(N,\left(N_{j},\psi_{j}\right)_{j\in J}\right)$
be ld-spaces. A surjective continuous ld-map $\theta:M\rightarrow N$
is called an \textit{ld-covering map} if there is a family $\left\{ \mathcal{O}_{l}:l\in L\right\} $
of open definable subspaces of $N$ such that

\begin{enumerate}[(i)]

\item $N=\bigcup_{l\in L}\mathcal{O}_{l}$,

\item the cover $\left\{ \mathcal{O}_{l}\cap N_{j}:l\in L\right\} $
of every $N_{j}$ admits a finite subcover, and

\item for every $l\in L$ and each connected component $C$ of $\theta^{-1}\left(\mathcal{O}_{l}\right)$,
the restriction $\theta\mid_{C}:C\rightarrow\mathcal{O}_{l}$ is a
definable homeomorphism (so in particular both $C$ and $\theta\mid_{C}$
are definable).

\end{enumerate}

We call $\left\{ \mathcal{O}_{l}:l\in L\right\} $ a $\theta$-\textit{admissible family}
of definable neighborhoods.

\end{dfn}

\begin{rem}\label{R:}
Let $\left(M,\left(M_{i},\phi_{i}\right)_{i\in I}\right)$ and $\left(N,\left(N_{j},\psi_{j}\right)_{j\in J}\right)$
be ld-spaces, and let $\theta:M\rightarrow N$ be a surjective continuous
ld-map. Then it is easy to prove that $\theta:M\rightarrow N$ is an ld-covering map if and
only if there is a family $\left\{ \mathcal{O}_{l}:l\in L\right\} $
of open definable subspaces of $N$  such that
\begin{enumerate}[(i)]
\item $N=\bigcup_{l\in L}\mathcal{O}_{l}$,

\item the cover $\left\{ \mathcal{O}_{l}\cap N_{j}:l\in L\right\} $
of every $N_{j}$ admits a finite subcover, and

\item for every $l\in L$, $\theta^{-1}\left(\mathcal{O}_{l}\right)$
is a disjoint union $\dot{\bigcup}_{i\in L_{l}}\mathcal{O}_{l,i}^{\prime}$
of open definable subspaces of $M$ such that for every $i\in L_{l}$
the restriction $\theta\mid_{\mathcal{O}_{l,i}^{\prime}}:\mathcal{O}_{l,i}^{\prime}\rightarrow\mathcal{O}_{l}$
is a definable homeomorphism  (so in particular both $\mathcal{O}_{l,i}^{\prime}$
and $\theta\mid_{\mathcal{O}_{l,i}^{\prime}}$ are definable).
\end{enumerate}
\end{rem}

\begin{comment}
($\Rightarrow$) This direction is clear since we can take the same
$\theta$-admissible family  $\left\{ \mathcal{O}_{l}:l\in L\right\} $
given by hypothesis, and note that the family of (path) connected
components of the open compatible subspace $\theta^{-1}\left(\mathcal{O}_{l}\right)$
is a family of open definable subspaces of $M$  each of which is
definably homeomorphic  to $\mathcal{O}_{l}$.

($\Leftarrow$) Let $\left\{ \mathcal{O}_{l}:l\in L\right\} $ be
the family of open definable subspaces of $N$  given by hypothesis
where $\theta^{-1}\left(\mathcal{O}_{l}\right)=\dot{\bigcup}_{i\in L_{l}}\mathcal{O}_{l,i}^{\prime}$
for every $l\in L$. Then \[
\left\{ \theta\left(C\right):C\,\textrm{is a connected component of }\mathcal{O}_{l,i}^{\prime},i\in L_{l},l\in L\right\} \]
is a $\theta$-admissible family of definable neighborhoods.
\end{proof}
\end{comment}

Now, we will prove that any ld-covering map between ld-spaces is closed for definable subspaces; notice that such a map is always open.

\begin{rem}\label{R:ClLim}
Let $\left(M,\left(M_{i},\phi_{i}\right)_{i\in I}\right)$ be an ld-space
and $X\subseteq M$ a definable space. If $y\in\textrm{Cl}\left(X\right)$,
then there is a definable map $g:\left(0,\epsilon\right)\rightarrow X$,
for some $\epsilon>0$, such that $\lim_{t\rightarrow0}g\left(t\right)=y$.
\end{rem}
\begin{proof}
Since $X$ is a definable subspace of $M$, then there is a finite
$J\subseteq I$ such that $X\subseteq\bigcup_{j\in J}M_{j}$. As $y\in\textrm{Cl}\left(X\right)$,
then $y\in\textrm{Cl}\left(X\cap M_{j}\right)$ for some $j\in J$.
Also, since $y\in M$, there is $k\in I$ such that $y\in M_{k}$.
Since $M_{k}$ is open in $M$ and $y\in\textrm{Cl}\left(X\cap M_{j}\right)$,
then $M_{k}\cap X\cap M_{j}\neq\emptyset$.

Because the definable spaces of an ld-space are closed under finite
intersections, then $\phi_{k}\left(M_{k}\cap X\cap M_{j}\right)$
is a definable set in $Z_{k}$. As $y\in\textrm{Cl}\left(M_{k}\cap X\cap M_{j}\right)$,
then $\phi_{k}\left(y\right)\in\textrm{Cl}\left(\phi_{k}\left(M_{k}\cap X\cap M_{j}\right)\right)$,
then, by \cite[Thm. 4.8]{P}, there is a definable map $\gamma:\left(0,\epsilon\right)\rightarrow\phi_{k}\left(M_{k}\cap X\cap M_{j}\right)$
such that $\lim_{t\rightarrow0}\gamma\left(t\right)=\phi_{k}\left(y\right)$.
Because $\phi_{k}$ is a homeomorphism between $M_{k}$ and $Z_{k}$,
$g=\phi_{k}^{-1}\circ\gamma:\left(0,\epsilon\right)\rightarrow M_{k}\cap X\cap M_{j}$
is a definable map such that $\lim_{t\rightarrow0}g\left(t\right)=y$.
\end{proof}

\begin{prop}\label{P:covisclosed}
Let $\left(M,\left(M_{i},\phi_{i}\right)_{i\in I}\right)$, $\left(N,\left(N_{j},\psi_{j}\right)_{j\in J}\right)$ be ld-spaces, a definable subspace $C\subseteq M$ closed in $M$, and $\theta:M\rightarrow N$ an ld-covering map. Then $\theta\left(C\right)\subseteq N$ is a definable subspace closed in $N$.
\end{prop}

\begin{proof}

We will show that if $y\in\textrm{Cl}\left(\theta\left(C\right)\right)$,
then $y\in\theta\left(C\right)$. As the image of a definable space
by an ld-map is a definable space, $\theta\left(C\right)$ is a definable
space. Because $y\in\textrm{Cl}\left(\theta\left(C\right)\right)$,
Remark \ref{R:ClLim} yields the existence of a definable map $g:\left(0,\epsilon\right)\rightarrow\theta\left(C\right)$
such that $\lim_{t\rightarrow0}g\left(t\right)=y$.

Now, since $\theta:M\rightarrow\mathcal{N}$ is an ld-covering map,
there is an open definable subspace $\mathcal{O}\subseteq N$ such that
$y\in\mathcal{O}$, $\theta^{-1}\left(\mathcal{O}\right)=\dot{\bigcup}_{j\in S}\mathcal{O}_{j}^{\prime}$, and each $\mathcal{O}_{j}^{\prime}$ is an open definable subspace
in $M$ homeomorphic to $\mathcal{O}$ by $\theta$.

Since $\lim_{t\rightarrow0}g\left(t\right)=y$ and $\mathcal{O}$
is an open neighborhood of $y$, there is $\delta>0$ such that $\delta\leq\epsilon$
and $g\left(\left(0,\delta\right)\right)\subseteq\mathcal{O}$. Without
loss of generality, we can assume that $\textrm{dom}\left(g\right)=\left(0,\delta\right)$.

Let $C^{\prime}=\left(\theta\mid_{C}\right)^{-1}\left(g\left(\left(0,\delta\right)\right)\right)$.
Since $C^{\prime}\subseteq\theta^{-1}\left(\mathcal{O}\right)=\dot{\bigcup}_{j\in S}\mathcal{O}_{j}^{\prime}$
and $C^{\prime}$ is definable, then saturation implies that $C^{\prime}\subseteq\mathcal{O}_{j_{1}}^{\prime}\cup\ldots\cup\mathcal{O}_{j_{k}}^{\prime}$
for some $k\in\mathbb{N}$. As $\left(0,\delta\right)=g^{-1}\left(\theta\left(C^{\prime}\right)\right)$,
then $\left(0,\delta\right)=g^{-1}\left(\theta\left(C^{\prime}\right)\right)=g^{-1}\left(\dot{\bigcup}_{1\leq i\leq k}\theta\left(\mathcal{O}_{j_{i}}^{\prime}\cap C^{\prime}\right)\right)=\dot{\bigcup}_{1\leq i\leq k}g^{-1}\circ\theta\left(\mathcal{O}_{j_{i}}^{\prime}\cap C^{\prime}\right)$.
Since $\theta\mid_{\mathcal{O}_{j}^{\prime}}:\mathcal{O}_{j}^{\prime}\rightarrow\mathcal{O}$
is a definable homeomorphism for every $j\in S$, any path $f$ in
$\mathcal{O}$ can be lifted through $\theta\mid_{\mathcal{O}_{j}^{\prime}}$
to a path $\widetilde{f}$ in $\mathcal{O}_{j}^{\prime}$. Therefore,
in particular, for each $i\in\left\{ 1,\ldots,k\right\} $, $\widetilde{g}_{j_{i}}=\theta\mid_{\mathcal{O}_{j_{i}}^{\prime}}^{-1}\circ g$
is a definable path in $\mathcal{O}_{j_{i}}^{\prime}\cap C^{\prime}$.
Hence, $\left(0,\delta\right)=\dot{\bigcup}_{1\leq i\leq k}\widetilde{g_{j_{i}}}^{-1}\left(\mathcal{O}_{j_{i}}^{\prime}\cap C^{\prime}\right)$.

By o-minimality, there are $j\in\left\{ 1,\ldots,k\right\} $ and
a positive $\epsilon^{\star}<\delta$ such that $\left(0,\epsilon^{\star}\right)\subseteq\widetilde{g}_{j_{i}}^{-1}\left(\mathcal{O}_{j_{i}}^{\prime}\cap C^{\prime}\right)$.
Let $x=\theta\mid_{\mathcal{O}_{j_{i}}^{\prime}}^{-1}\left(y\right)$.
Since $g\mid_{\left(0,\epsilon^{\star}\right)}\left(t\right)\underset{t\rightarrow0}{\longrightarrow}y$,
then the lifting $\widetilde{g}_{j_{i}}\mid_{\left(0,\epsilon^{\star}\right)}\left(t\right)\underset{t\rightarrow0}{\longrightarrow}x$.
So, $x\in\textrm{Cl}\left(C\right)$. But $C$ is closed in $\mathcal{U}$,
so $x\in C$; namely, $y\in\theta\left(C\right)$. Then the image
by $\theta$ of any definable closed subspace of $M$ is closed in
$N$. This ends the proof of Proposition \ref{P:covisclosed}.
\end{proof}

With the previous proposition we can prove the existence of a homeomorphism between simply connected definable spaces as a restriction of a given ld-covering map as we see in the following proposition.

\begin{prop}\label{P:CovPties1}

Let $M$, $N$ be ld-spaces and $\theta:M\rightarrow N$ an ld-covering
map. Let $Y\subseteq N$ be a compatible subspace in $N$, then $\theta\mid_{\theta^{-1}\left(Y\right)}:\theta^{-1}\left(Y\right)\rightarrow Y$
is an ld-covering map of ld-spaces. If, moreover, $Y$ is definable, $n_0\in Y$, and $m_0\in\theta^{-1}\left(n_0\right)$,
then there is a definable subspace $W\subseteq M$ open in $\theta^{-1}\left(Y\right)$
such that $m_0\in W$ and the following hold.

\begin{enumerate}[(i)]

\item $\theta\mid_{W}:W\rightarrow Y$ is a definable covering map
of ld-spaces.

\item If in addition $Y$ is simply connected, then $\theta\mid_{W_{m_0}}:W_{m_0}\rightarrow Y$
is a homeomorphism of definable spaces where $W_{m_0}$ is the connected component
of $m_0$ in $W$.

\end{enumerate}

\end{prop}

\begin{proof}

Since the preimage of a compatible subspace by an ld-map is a compatible
subspace, $\theta^{-1}\left(Y\right)$ is a compatible subset of $M$.
As $\theta\mid_{\theta^{-1}\left(Y\right)}:\theta^{-1}\left(Y\right)\rightarrow Y$ is a continuous surjection,
it only remains to show the existence of a $\theta\mid_{\theta^{-1}\left(Y\right)}$-admissible
family of definable neighborhoods. Let $\left\{ \mathcal{O}_{i}\right\} _{i\in L}$
be a $\theta$-admissible family of definable neighborhoods such that
$\theta^{-1}\left(\mathcal{O}_{i}\right)=\dot{\bigcup}_{j\in S_{i}}\mathcal{O}_{i_{j}}^{\prime}$ and $\mathcal{O}_{i_{j}}^{\prime}$ is ld-homeomorphic to $\mathcal{O}_{i}$ by $\theta$ for any $i\in L$, $j\in S_i$. Then it is easy to see that $\left\{ \mathcal{O}_{i}\cap Y\right\} _{i\in L}$
is a $\theta\mid_{\theta^{-1}\left(Y\right)}$-admissible family of
definable neighborhoods.

Now, assume that $Y$ is also a definable space. Following, we will
prove (i). Let $\left\{ \mathcal{O}_{i}\cap Y\right\} _{i\in L}$
be the above $\theta\mid_{\theta^{-1}\left(Y\right)}$ -admissible
family of definable neighborhoods for $\theta\mid_{\theta^{-1}\left(Y\right)}:\theta^{-1}\left(Y\right)\rightarrow Y$.
Hence, the definability of $Y$ and the saturation of the model imply
that there is $s\in\mathbb{N}$ such that $Y=\bigcup_{1\leq i\leq s}\mathcal{O}_{i}\cap Y$.
For each $i\in\left\{ 1,\ldots,s\right\} $ fix an arbitrary finite
nonempty subset $S_{i}^{\prime}\subseteq S_{i}$ such that if $e_{n_0}\in\mathcal{O}_{i}\cap Y$,
then there is $j\in S_{i}^{\prime}$ such that $e_{m_0}\in\mathcal{O}_{i_{j}}^{\prime}\cap\theta^{-1}\left(Y\right)$.
Let $W=\bigcup \left\{ \mathcal{O}_{i_{j}}^{\prime}\cap\theta^{-1}\left(Y\right):i\in\left\{ 1,\ldots,s\right\} ,\, j\in S_{i}^{\prime}\right\} $,
which is open in $\theta^{-1}\left(Y\right)$. Then $\left\{ \mathcal{O}_{i}\cap Y:i\in\left\{ 1,\ldots,s\right\} \right\} $
is a $\theta\mid_{W}$-admissible family of definable neighborhoods.
So $\theta\mid_{W}:W\rightarrow Y$ is a definable covering map of
ld-spaces.

For (ii), first we will prove that $\theta\mid_{W_{m_0}}:W_{m_0}\rightarrow Y$
is a definable covering map of ld-spaces, this is the next claim.

\begin{claim}\label{C:covinW0}

Let $W_{m_0}$ be the connected component of $m_0$ in $W$. Then

\begin{enumerate}[(i)]

\item $\theta\mid_{W_{m_0}}:W_{m_0}\rightarrow Y$ is surjective.

\item There is a $\theta\mid_{W_{m_0}}$-admissible family of definable
neighborhoods.

\end{enumerate}

Therefore, $\theta\mid_{W_{m_0}}:W_{m_0}\rightarrow Y$ is a definable
covering map of ld-spaces.

\end{claim}

\begin{proof}

(i) By Fact 4.2 of \cite{BarOter}, $W_{m_0}$ is a clopen definable
subset of $W$. By Proposition \ref{P:covisclosed}, $\theta\left(W_{m_0}\right)$
is a definable space clopen in $Y$, but $Y$ is connected, so $\theta\left(W_{m_0}\right)=Y$;
i.e., $\theta$ is surjective.

(ii) The same $\theta\mid_{W}$-admissible family of definable neighborhoods
$\left\{ \mathcal{O}_{i}\cap Y:i\in\left\{ 1,\ldots,s\right\} \right\} $
works for $\theta\mid_{W_{m_0}}$ because if $C$ is a connected component
of $\theta\mid_{W}^{-1}\left(\mathcal{O}_{i}\cap Y\right)$ in $W$,
then $C$ is either entirely contained in $W_{m_0}$ or is disjoint
from $W_{m_0}$. Therefore, $C$ is homeomorphic by $\theta\mid_{W_{m_0}}$
with $\mathcal{O}_{i}\cap Y$.

From (i) and (ii), $\theta\mid_{W_{m_0}}:W_{m_0}\rightarrow Y$ is a definable
covering map of ld-spaces.

\end{proof}

Since $Y$ is simply connected, \cite[Remark 3.8]{EdPanPre2013} implies
that there is an ld-covering map $\beta:Y\rightarrow W_{m_0}$ such that
$\textrm{id}=\theta\mid_{W_{m_0}}\circ\beta$, then $\theta\mid_{W_{m_0}}:W_{m_0}\rightarrow Y$
is a definable homeomorphism.

\end{proof}

\section{The o-minimal universal covering homomorphism of a locally definable group}\label{S:3}

This section is devoted to introduce the notion and properties of locally definable covering homomorphism and o-minimal universal covering homomorphism.

\begin{dfn}\label{D:}
Let $\mathcal{U}$, $\mathcal{V}$ be locally definable groups. An ld-covering map $\theta:\mathcal{U}\rightarrow\mathcal{V}$ that is also a homomorphism is called a \textit{locally definable covering homomorphism}. As before, $\left\{ U_{i}\right\} _{i\in I}$
is called a $\theta$\textit{-admissible family of definable neighborhoods}.

Two locally definable covering homomorphisms $\theta:\mathcal{U}\rightarrow\mathcal{V}$,
$\theta^{\prime}:\mathcal{U}^{\prime}\rightarrow\mathcal{V}$ are
called \textit{equivalent} if there are locally definable covering
homomorphisms $\beta:\mathcal{U}\rightarrow\mathcal{U}^{\prime}$
and $\beta^{\prime}:\mathcal{U}^{\prime}\rightarrow\mathcal{U}$ such
that $\theta=\theta^{\prime}\circ\beta$ and $\theta^{\prime}=\theta\circ\beta^{\prime}$,
so the following diagram commutes.
\[
\xymatrix@C=30pt{\mathcal{U}\ar[d]_{\theta}\ar@{-->}@/^/[r]^{\beta} & \mathcal{U}^{\prime}\ar[dl]^{\theta^{\prime}}\ar@{-->}@/^/[l]^{\beta^{\prime}}\\
\mathcal{V}
}
\]
In general, in our diagrams the regular arrows are maps whose existence is assumed, and the dashed arrows are maps whose existence is asserted. The inclusion map is denoted by $i$.
\end{dfn}

\begin{comment}
Let $\mathcal{U}$, $\mathcal{V}$ be locally definable groups. A
surjective locally definable homomorphism $\theta:\mathcal{U}\rightarrow\mathcal{V}$
is called a \textit{locally definable covering homomorphism} if there
is a family $\left\{ U_{i}\right\} _{i\in I}$ of open definable subsets
of $\mathcal{V}$ such that $\mathcal{V}=\bigcup_{i\in I}U_{i}$,
and for every $i\in I$ $\theta^{-1}\left(U_{i}\right)$ is a disjoint
union of open definable subsets of $\mathcal{U}$ each of which is
homeomorphic to $U_{i}$ by $\theta$. $\left\{ U_{i}\right\} _{i\in I}$
is called a $\theta$\textit{-admissible family of definable neighborhoods}.
\end{comment}

\begin{fact}\cite[Theorem 3.6]{ED05}\label{F:Thm3.6}
Let $\theta:\mathcal{U}\rightarrow\mathcal{V}$ be a surjective locally
definable homomorphism between locally definable groups. If $\ker\left(\theta\right)$
has dimension zero, then $\theta:\mathcal{U}\rightarrow\mathcal{V}$
is a locally definable covering homomorphism.
\end{fact}

\begin{dfn}\label{D:DefUnivCov}
Let $\mathcal{V}$ be a connected locally definable group. A locally
definable covering homomorphism $\theta:\mathcal{U}\rightarrow\mathcal{V}$
with $\mathcal{U}$ connected is called an \textit{o-minimal universal covering homomorphism} of $\mathcal{V}$
if for every locally definable covering homomorphism $\pi:\mathcal{Z}\rightarrow\mathcal{V}$
with $\mathcal{Z}$ connected, there exists a locally definable covering
homomorphism $\beta:\mathcal{U}\rightarrow\mathcal{Z}$ such that $\theta=\pi\circ\beta$. In this case $\mathcal{Z}$ is called an \textit{o-minimal universal covering group} of $\mathcal{V}$.
\end{dfn}

Note that if there are two o-minimal universal covering homomorphisms $\theta:\mathcal{U}\rightarrow\mathcal{V}$ and
$\theta^{\prime}:\mathcal{U}^{\prime}\rightarrow\mathcal{V}$ of a connected locally definable group $\mathcal{V}$, then there exist locally definable covering
homomorphisms $\beta:\mathcal{U}\rightarrow\mathcal{U}^{\prime}$
and $\beta^{\prime}:\mathcal{U}^{\prime}\rightarrow\mathcal{U}$ such
that $\theta=\theta^{\prime}\circ\beta$ and $\theta^{\prime}=\theta\circ\beta^{\prime}$. Therefore, if $\mathcal{V}$ has an o-minimal universal covering homomorphism, then it is unique up to equivalent locally definable covering homomorphisms. Thus, we can say ``the'' o-minimal universal covering homomorphism of $\mathcal{V}$, and sometimes we denote the o-minimal universal covering group of $\mathcal{V}$ by $\widetilde{\mathcal{V}}$.

In \cite{EdPan} Edmundo and Eleftheriou constructed a locally definable covering homomorphism $\theta:\mathcal{U}\rightarrow\mathcal{V}$ for a given connected locally definable group $\mathcal{V}$ that satisfies the definition of an o-minimal universal covering homomorphism of $\mathcal{V}$ (Def. \ref{D:DefUnivCov}) (so the o-minimal universal covering homomorphism of $\mathcal{V}$ exists), and they showed the following.

\begin{fact}\label{F:EdThm3.11}\cite[Thm. 3.11]{EdPan}
For a connected locally definable group $\mathcal{V}$, the kernel of its o-minimal universal covering homomorphism is isomorphic, as abstract groups, to the o-minimal fundamental group $\pi_{1}\left(\mathcal{V}\right)$.
\end{fact}

\begin{fact}\label{F:UnvCovIFFisS.C.}\cite[Remark 3.8]{EdPanPre2013}
A locally definable covering homomorphism $\pi:\mathcal{U}\rightarrow\mathcal{V}$ between connected locally definable groups $\mathcal{U}$ and $\mathcal{V}$ is the o-minimal universal covering homomorphism of $\mathcal{V}$ if and only if $\pi_{1}\left(\mathcal{U}\right)=\left\{ 0\right\}$.
\end{fact}

\begin{rem}\label{R:PropCover}

Let $\mathcal{U}$, $\mathcal{V}$ be connected locally definable
groups, and $\theta:\mathcal{U}\rightarrow\mathcal{V}$ a locally
definable covering homomorphism. Then

\begin{enumerate}[(i)]

\item $\mathcal{U}$ is abelian if and only if $\mathcal{V}$ is
abelian.

\item Assume that $\mathcal{V}$ is abelian. Then $\mathcal{U}$ is divisible
if and only if $\mathcal{V}$ is divisible.

\end{enumerate}

\end{rem}

\begin{proof}

(i) Clearly, if $\mathcal{U}$ is abelian, by the surjectiveness of
$\theta$, $\mathcal{V}$ is abelian. Now, assume that $\mathcal{V}$
is abelian, and let $\pi:\widetilde{\mathcal{V}}\rightarrow\mathcal{V}$
be the o-minimal universal covering homomorphism of $\mathcal{V}$.
Then there is a locally definable covering
homomorphism $\beta:\widetilde{\mathcal{V}}\rightarrow\mathcal{U}$
such that $\pi=\theta\circ\beta$. Since $\mathcal{V}$ is abelian,
so is $\widetilde{\mathcal{V}}$, then, by going through $\beta$, $\mathcal{U}$ is also abelian.

(ii) It is clear that if $\mathcal{U}$ is divisible, by the surjectiveness
of $\theta$, $\mathcal{V}$ is divisible. The another implication needs
the abelianness of the groups, and it is \cite[Proposition 5.13]{BerarEdMamino}.

\end{proof}

\begin{fact}\label{F:UniCovIsTF}\cite[Proposition 5.14]{BerarEdMamino}
The o-minimal universal covering group of a connected abelian divisible locally definable group is divisible and torsion free.
\end{fact}

\begin{claim}\label{C:tf-is-sc}

Let $\mathcal{U}$ be a connected locally definable group covering
an abelian connected definable group $G$. If $\mathcal{U}$ is torsion
free, then $\mathcal{U}$ is simply connected.

\end{claim}

\begin{proof}

Since $G$ is an abelian (definably) connected definable group, then $G$
is divisible (see, e.g., the proof of \cite[Theorem 2.1]{OtEd04}).
Then $\mathcal{U}$ is also abelian and divisible, by Remark \ref{R:PropCover}.
So the map $p_{k}:\mathcal{U}\rightarrow\mathcal{U}:x\mapsto x^{k}$
is a bijective locally definable homomorphism for any $k\in\mathbb{N}$,
so in particular $p_{k}$ is a locally definable covering homomorphism.
Thus, by \cite[Corollary 6.12]{BarOter} or \cite[Proposition 4.6]{ED05},
the induced map $p_{k,*}:\pi_{1}\left(\mathcal{U}\right)\rightarrow\pi_{1}\left(\mathcal{U}\right):\left[\gamma\right]\mapsto\left[p_{k}\circ\gamma\right]$
is an injective homomorphism; therefore, the $k$-torsion group $\mathcal{U}\left[k\right]$
of $\mathcal{U}$ satisfies that $\left\{ 0\right\} =\mathcal{U}\left[k\right]=\ker\left(p_{k,*}\right)\cong\pi_{1}\left(\mathcal{U}\right)/p_{k,*}\left(\pi_{1}\left(\mathcal{U}\right)\right)$.
Then, $\pi_{1}\left(\mathcal{U}\right)=\left(\pi_{1}\left(\mathcal{U}\right)\right)^{k}$
for every $k\in\mathbb{N}$, thus $\pi_{1}\left(\mathcal{U}\right)$
is a divisible group.

Now, let $\theta:\mathcal{U}\rightarrow G$ be a locally definable
covering homomorphism, and $\theta_{*}:\pi_{1}\left(\mathcal{U}\right)\rightarrow\pi_{1}\left(G\right)$
its induced injective homomorphism, so $\theta_{*}\left(\pi_{1}\left(\mathcal{U}\right)\right)$
is a divisible subgroup of $\pi_{1}\left(G\right)$. By \cite[Theorem 2.1]{OtEd04},
there is $s\in\mathbb{N}$ such that $\pi_{1}\left(G\right)\cong\mathbb{Z}^{s}$,
then the only possible divisible subgroup of $\pi_{1}\left(G\right)$
is the trivial one, so $\pi_{1}\left(\mathcal{U}\right)=\left\{ 0\right\}$.

\end{proof}

From Fact \ref{F:UniCovIsTF} and Claim \ref{C:tf-is-sc}, we have that if $G$ is a connected abelian definable group, $G$ is torsion free if and only if $G$ is simply connected.

\begin{cor}\label{C:tf-is-UniCov}
Let $\mathcal{U}$ be a connected torsion free locally definable group,
$G$ an abelian connected definable group, and $\theta:\mathcal{U}\rightarrow G$
a locally definable covering homomorphism. Then $\theta:\mathcal{U}\rightarrow G$
is the o-minimal universal covering homomorphism of $G$.
\end{cor}
\begin{proof}
By \cite[Remark 3.8]{EdPanPre2013} and Claim \ref{C:tf-is-sc}, $\mathcal{U}$ is simply connected. So, by Fact \ref{F:UnvCovIFFisS.C.}, $\theta:\mathcal{U}\rightarrow G$
is the o-minimal universal covering homomorphism of $G$.
\end{proof}

\section{Abelian definably generated groups, convex sets, and covers of definable groups}\label{S:4}

In this section we present some properties of the abelian $\bigvee$-definable groups in relation to their smallest type-definable subgroup of index smaller than $\kappa$, if it exists, and to some generic subsets and convex sets.

Note that if $\mathcal{U}$ is a connected $\bigvee$-definable group with $\mathcal{U}^{00}$, then $\mathcal{U}$ has a definable left-generic set, thus, by Fact 2.3 in \cite{PPant12}, $\mathcal{U}$ is definably generated, and hence locally definable.

In the first part of this section, we point out some central facts about the existence of $\mathcal{U}^{00}$
for an abelian definably generated group $\mathcal{U}$ as well as
necessary and sufficient conditions for being a cover of a definable
group. The first of these facts gathers Proposition 3.5 and Theorem 3.9 of Peterzil and Eleftheriou's work in \cite{PPant12}.

\begin{fact}\label{F:EqvU00-I}\cite{PPant12}
Let $\mathcal{U}$ be a connected abelian definably generated group
of dimension $d$. Then:

\begin{enumerate}[(i)]

\item $\mathcal{U}$ covers a definable group if and only if the
subgroup $\mathcal{U}^{00}$ exists if and only if $\mathcal{U}$
contains a definable generic set.

\item If $\mathcal{U}^{00}$ exists, then $\mathcal{U}^{00}$ is
torsion free, $\mathcal{U}$ and $\mathcal{U}^{00}$ are divisible,
and $\mathcal{U}/\mathcal{U}^{00}$ is a Lie group isomorphic, as
a topological group, to $\mathbb{R}^{k}\times\mathbb{T}^{r}$ for
some $k,r\in\mathbb{N}$ with $k+r\leq d$, where $\mathbb{T}$ is
the circle group.

\end{enumerate}

\end{fact}

\begin{dfn}\label{D:convexHull} \cite[Def. 5.3]{BerarEdMamino}
Let $G$ be an abelian group and $X\subseteq G$.

\begin{enumerate}[(i)]

\item $X$ is called \textit{convex} if for every $a,b\in X$ and
$n,m\in\mathbb{N}$, not both null, $X$ contains every solution $x\in X$
of the equation $x^{m+n}=a^{m}b^{n}$.

\item The \textit{convex hull} $ch\left(X\right)$ of $X$ is the set
of all $x\in G$ such that $x^{n}=a_{1}\cdots a_{n}$ for some $n\in\mathbb{N}$
and some $a_{1},\ldots,a_{n}\in X$ not necessarily distinct.

\item A locally definable abelian group $\mathcal{U}$ has \textit{definably bounded convex hulls}
if for all definable $X\subseteq\mathcal{U}$, there is a definable
$Y\subseteq\mathcal{U}$ such that $ch\left(X\right)\subseteq Y$.

\end{enumerate}

\end{dfn}

If $\mathcal{U}$ is a divisible torsion free abelian group, then it is easy to prove that $X\subseteq\mathcal{U}$ is convex if and only if $\prod_{n}X=X^{n}$ for every $n\in\mathbb{N}$.

\begin{fact}\label{F:EqvU00-II}\cite[Theorem 5.6]{BerarEdMamino}
Let $\mathcal{U}$ be a connected abelian definably generated group.
The following are equivalent:

\begin{enumerate}[(i)]

\item $\mathcal{U}$ covers a definable group.

\item For every definable $X\subseteq\mathcal{U}$, there is a definable
$Y\subseteq\mathcal{U}$ such that $\prod_{n}X\subseteq Y^{n}$ for all $n\in\mathbb{N}$.

\item $\mathcal{U}$ is divisible and has definably bounded convex
hulls.

\end{enumerate}
\end{fact}

The second part of this section is devoted to prove Proposition \ref{P:ImpProp}.

\begin{claim}\label{C:TK-C3}
Let $L$ be a topological group isomorphic, as a topological group,
to $\mathbb{R}^{k}\times\mathbb{T}^{r}$ for some $k,r\in\mathbb{N}$, where $\mathbb{T}$ is the circle group. Let $C\subseteq L$
be a compact neighbourhood of the identity element $e_{L}$ of $L$.
Then there is an increasing sequence $\left\{ n_{i}\right\} _{i}\subseteq\mathbb{N}$ such
that $C^{n_{i}}\subseteq C^{n_{i+1}}$ for every $i\in\mathbb{N}$,
and $L=\bigcup_{i\in\mathbb{N}}C^{n_{i}}$.
\end{claim}
\begin{proof}
First, note that in $\mathbb{R}^{k}\times\mathbb{T}^{r}$ every compact neighbourhood
$Y\subseteq\mathbb{R}^{k}\times\mathbb{T}^{r}$ of the identity element $e$ of $\mathbb{R}^{k}\times\mathbb{T}^{r}$, there is a neighbourhood
$X\subseteq Y$ of $e$ such that $\mathbb{R}^{k}\times\mathbb{T}^{r}=\bigcup_{n\in\mathbb{N}}X^{n}$
and $X^{n}\subseteq X^{n+1}$. Therefore, as $L$ and $\mathbb{R}^{k}\times\mathbb{T}^{r}$ are isomorphic as a topological groups, then there is a neighbourhood $\mathcal{O}\subseteq C$ of $e_{L}$ such that $L=\bigcup_{n\in\mathbb{N}}\mathcal{O}^{n}$,
and $\mathcal{O}^{n}\subseteq\mathcal{O}^{n+1}$ for every $n\in\mathbb{N}$.

Let us define the sequence $\left\{ n_{i}\right\} _{i}\subseteq\mathbb{N}$
inductively as follows.

Let $n_{1}=1$. Let us assume that $n_{i-1}$ is defined for $i\geq2$.
Since $C$ is compact, $C^{n_{i-1}}\cup C^{i}$ is compact, so $C^{n_{i-1}}\cup C^{i}\subseteq\bigcup_{n\in\mathbb{N}}\mathcal{O}^{n}$
yields the existence of finitely many natural numbers $i_{1},\ldots,i_{s}$
such that $C^{n_{i-1}}\cup C^{i}\subseteq\mathcal{O}^{i_{1}}\cup\ldots\cup\mathcal{O}^{i_{s}}$.
As $\mathcal{O}^{n}\subseteq\mathcal{O}^{n+1}$ for every $n\in\mathbb{N}$,
then $C^{n_{i-1}}\cup C^{i}\subseteq\mathcal{O}^{n_{i}}$ where $n_{i}=\max\left\{ i_{1},\ldots,i_{s},n_{i-1}\right\} $.

Finally, by the definition of the $n_{i}$'s and $\mathcal{O}\subseteq C$,
it follows directly that $C^{n_{i}}\subseteq C^{n_{i+1}}$ for every
$i\in\mathbb{N}$, and $L=\bigcup_{i\in\mathbb{N}}C^{n_{i}}$.
\end{proof}

\begin{prop}\label{P:ImpProp}
Let $\mathcal{U}$ be a connected abelian $\bigvee$-definable group
such that   $\mathcal{U}^{00}$ exists. Let $X\subseteq\mathcal{U}$
be a definable set such that $\mathcal{U}^{00}\subseteq X$ and $Z\subseteq\mathcal{U}$ a definable set. Then
\begin{enumerate}[(i)]
\item $\mathcal{U}=\bigcup_{n\in\mathbb{N}}X^{n}$.

\item There is $k\in\mathbb{N}$ such that $Z\subseteq X^{k}$.

\item There is $k\in\mathbb{N}$ such that the convex hull $\textrm{ch}\left(Z\right)$
of $Z$ is contained in $X^{k}$. If, moreover, $\mathcal{U}$ is
torsion free, then $\textrm{ch}\left(Z^{\frac{1}{k}}\right)\subseteq X$.

\end{enumerate}

\end{prop}

\begin{proof}

Let $L$ denote the group $\mathcal{U}/\mathcal{U}^{00}$, let $\pi:\mathcal{U}\rightarrow L$
be the quotient homomorphism, and consider $L$ as the locally compact topological space given by the logic topology (see \cite[Lemma 7.5]{HPePiNIP}). By \cite[Thm. 3.9]{PPant12}, $L$ is isomorphic, as a topological group, to $\mathbb{R}^{r_{1}}\times\mathbb{T}^{r_{2}}$ for some $r_{1},r_{2}\in\mathbb{N}$.

As $\mathcal{U}^{00}\subseteq X$, saturation yields the existence of a definable $Y\subseteq X$ such
that $\mathcal{U}^{00}\subseteq Y\subseteq Y\cdot Y\subseteq X$. Thus, by \cite[Fact 2.3(2)]{PPant12}, $Y$ generates $\mathcal{U}$. Furthermore, $\pi^{\prime}\left(Y\right)=\left\{ l\in L:\pi^{-1}\left(l\right)\subseteq Y\right\}$ is an open neighbourhood of the identity element $e_{L}$ of $L$.
Therefore, $\pi\left(Y\right)$ is a compact connected neighbourhood
of $e_{L}$ in $L$ and generates $L$.

Claim \ref{C:TK-C3} yields the existence of an increasing
sequence $\left\{ n_{i}\right\} _{i\in\mathbb{N}}\subseteq\mathbb{N}$
such that $\pi\left(Y\right)^{n_{i}}\subseteq\pi\left(Y\right)^{n_{i+1}}$
for every $i\in\mathbb{N}$ and $L=\bigcup_{i\in\mathbb{N}}\pi\left(Y\right)^{n_{i}}=\pi\left(\bigcup_{i\in\mathbb{N}}Y^{n_{i}}\right)$.
Hence, \[ \mathcal{U}=\bigcup_{i\in\mathbb{N}}Y^{n_{i}}\cdot\mathcal{U}^{00}=\bigcup_{i\in\mathbb{N}}\left(Y\cdot\mathcal{U}^{00}\right)^{n_{i}}\subseteq\bigcup_{i\in\mathbb{N}}X^{n_{i}}\subseteq\bigcup_{n\in\mathbb{N}}X^{n}.\]
This gives us (i).

Since $\pi\left(X\right)$ is a compact set in $L$ and $L=\bigcup_{i\in\mathbb{N}}\pi\left(Y\right)^{n_{i}}$,
there are $k_{1},\ldots,k_{s}\in\left\{ n_{i}\right\} _{i\in\mathbb{N}}$
such that $\pi\left(X\right)\subseteq\pi\left(Y\right)^{k_{1}}\cup\ldots\cup\pi\left(Y\right)^{k_{s}}$.
As $\pi\left(Y\right)^{n_{i}}\subseteq\pi\left(Y\right)^{n_{i+1}}$,
then $\pi\left(X\right)\subseteq\pi\left(Y\right)^{k^{\star}}$ where
$k^{\star}=\max\left\{ k_{1},\ldots,k_{s}\right\} $. Therefore, \[
X\subseteq X\cdot\mathcal{U}^{00}\subseteq Y^{k^{\star}}\cdot\mathcal{U}^{00}\subseteq\left(Y\cdot Y\right)^{k^{\star}}\subseteq X^{k^{\star}}.\]
By (i) and saturation, if $Z\subseteq\mathcal{U}$ is a definable
set, then there are $l_{1},\ldots,l_{m}\in\mathbb{N}$ such that $Z\subseteq X^{l_{1}}\cup\ldots\cup X^{l_{m}}$,
then $Z\subseteq X^{k^{\star\star}}$ where $k^{\star\star}=\max\left\{ l_{1}k^{\star},\ldots,l_{m}k^{\star}\right\} $,
which yields (ii).

Finally, let us prove (iii). By Lemma 3.7 in \cite{PPant12}, $\mathcal{U}^{00}$
exists if and only if $\mathcal{U}$ covers a definable group, and
by Theorem 5.6 in \cite{BerarEdMamino}, if and only if $\mathcal{U}$
has definably bounded convex hulls; i.e., for every definable $Z^{\prime}\subseteq\mathcal{U}$
there is a definable $W\subseteq\mathcal{U}$ containing the convex
hull $\textrm{ch}\left(Z^{\prime}\right)$ of $Z^{\prime}$. Then, there is a definable set $W$ such that
$\textrm{ch}\left(Z\right)\subseteq W\subseteq\mathcal{U}$, and (ii)
yields the existence of $k\in\mathbb{N}$ such that $W\subseteq X{}^{k}$,
then $\textrm{ch}\left(Z\right)\subseteq X^{k}$. Since $\mathcal{U}^{00}$
exists, Proposition 3.5 in \cite{PPant12} implies that $\mathcal{U}$
is divisible. If in addition $\mathcal{U}$ is torsion free, then
the map $x\mapsto x^{k}:\mathcal{U}\rightarrow\mathcal{U}$ is a group
isomorphism for every $k\in\mathbb{N}$, so if $\textrm{ch}\left(Z\right)\subseteq X^{k}$,
then $\textrm{ch}\left(Z^{\frac{1}{k}}\right)\subseteq X$.
\end{proof}

\section{Local homomorphisms and generic sets: some technical propositions}\label{S:5}

Below we prove some technical results that will be applied in the proofs of Theorems \ref{T:Z_t-f}, and \ref{T:UnivCovAbLD}.

\begin{prop}\label{P:R1}
Let $\mathcal{Z}$ and $\mathcal{V}$ be locally definable groups
such that $\mathcal{Z}^{00}$ exists. Let $W\subseteq\mathcal{Z}$
be a definable set such that $\mathcal{Z}^{00}\subseteq W$, and $\theta:W\rightarrow\mathcal{V}$
be a definable local homomorphism. Then
\begin{enumerate}[(i)]

\item there is a definable symmetric set $W^{\prime}\subseteq W$
such that $\mathcal{Z}^{00}\subseteq W^{\prime}\subseteq\prod_{4}W^{\prime}\subseteq W$
and $\theta\left(W^{\prime}\right)$ is generic in $\left\langle \theta\left(W^{\prime}\right)\right\rangle$.

\item $\theta\left(\mathcal{Z}^{00}\right)$ is
a type-definable subgroup of $\left\langle \theta\left(W^{\prime}\right)\right\rangle $
of index less than $\kappa$, and hence $\left\langle \theta\left(W^{\prime}\right)\right\rangle ^{00}\subseteq\theta\left(\mathcal{Z}^{00}\right)\subseteq\theta\left(W^{\prime}\right)$.
\end{enumerate}
\end{prop}
\begin{proof}
(i) As $\mathcal{Z}^{00}\subseteq W$, saturation implies that  there is a definable
symmetric $W^{\prime}\subseteq W$ such that $\mathcal{Z}^{00}\subseteq W^{\prime}\subseteq\prod_{4}W^{\prime}\subseteq W$.
Since $W^{\prime}$ is generic in $\mathcal{Z}$ and the structure is $\kappa$-saturated (with $\kappa\geq\aleph_{1}$), then $W^{\prime}W^{\prime}\subseteq\bigcup_{i<\aleph_{1}}w_{i}W^{\prime}$
for some $\left\{ w_{i}\right\} _{i<\aleph_{1}}\subseteq\mathcal{Z}$.

Let $I=\left\{ i<\aleph_{1}:W^{\prime}W^{\prime}\cap w_{i}W^{\prime}\neq\emptyset\right\} $, and $i\in I$. If $xy=w_{i}z$ with $x,y,z\in W^{\prime}$, then $w_{i}=xyz^{-1}\in\prod_{3}W^{\prime}$,
thus $w_{i}W^{\prime}\subseteq\prod_{4}W^{\prime}\subseteq W$. Therefore,
\[ \theta\left(W^{\prime}W^{\prime}\right)=\theta\left(W^{\prime}\right)\theta\left(W^{\prime}\right)\subseteq\bigcup_{i\in I}\theta\left(w_{i}\right)\theta\left(W^{\prime}\right)\subseteq\left\langle \theta\left(W^{\prime}\right)\right\rangle ,\]
and $\theta\left(w_{i}\right)\in\left\langle \theta\left(W^{\prime}\right)\right\rangle $
for $i\in I$. Hence, $\theta\left(W^{\prime}\right)\theta\left(W^{\prime}\right)$ is covered by $\theta\left(W^{\prime}\right)$ by $<\aleph_{1}$ group translates. It implies that $\theta\left(W^{\prime}\right)$ is a definable generic subset in $\left\langle \theta\left(W^{\prime}\right)\right\rangle $.

(ii) We will see that $\theta\left(\mathcal{Z}^{00}\right)$ is a
type-definable subgroup of $\left\langle \theta\left(W^{\prime}\right)\right\rangle $
of index less than $\kappa$. By saturation, $\theta\left(\mathcal{Z}^{00}\right)$
is a type-definable set. Now, as $\left[\mathcal{Z}:\mathcal{Z}^{00}\right]<\kappa$, $W^{\prime}\subseteq\bigcup_{j<\kappa}b_{j}\mathcal{Z}^{00}$
with $\left\{ b_{j}\right\} _{j<\kappa}\subseteq\mathcal{Z}$. Let
$J=\left\{ j<\kappa:W^{\prime}\cap b_{j}\mathcal{Z}^{00}\neq\emptyset\right\} $.
Then, if $j\in J$ and $x=b_{j}z$ with $x\in W^{\prime}$, $z\in\mathcal{Z}^{00}$,
then $b_{j}=xz^{-1}\in W^{\prime}\mathcal{Z}^{00}\subseteq\prod_{2}W^{\prime}$,
so $\theta\left(b_{j}\right)\in\prod_{2}\theta\left(W^{\prime}\right)$.
Thus, \[ \theta\left(W^{\prime}\right)\subseteq\bigcup_{j\in J}\theta\left(b_{j}\right)\theta\left(\mathcal{Z}^{00}\right),\:\textrm{and}\: \theta\left(b_{j}\right)\in\left\langle \theta\left(W^{\prime}\right)\right\rangle .\]

In addition, by (i), $\left\langle \theta\left(W^{\prime}\right)\right\rangle \subseteq\bigcup_{i<\kappa}v_{i}\theta\left(W^{\prime}\right)$
for some $\left\{ v_{i}\right\} _{i<\kappa}\subseteq\left\langle \theta\left(W^{\prime}\right)\right\rangle $.
Then,   \[ \left\langle \theta\left(W^{\prime}\right)\right\rangle \subseteq\bigcup_{i<\kappa}v_{i}\bigcup_{j\in J}\theta\left(b_{j}\right)\theta\left(\mathcal{Z}^{00}\right)=\bigcup_{i<\kappa,j\in J}v_{i}\theta\left(b_{j}\right)\theta\left(\mathcal{Z}^{00}\right).\]
Hence, $\left[\left\langle \theta\left(W^{\prime}\right)\right\rangle :\theta\left(\mathcal{Z}^{00}\right)\right]<\kappa$.

Note that since $\theta\left(\mathcal{Z}^{00}\right)$ is a type-definable subgroup of $\left\langle \theta\left(W^{\prime}\right)\right\rangle$ of index $<\kappa$, then $\left\langle \theta\left(W^{\prime}\right)\right\rangle ^{00}$ exists (see \cite[Prop. 7.4]{HPePiNIP}), and thus $\left\langle \theta\left(W^{\prime}\right)\right\rangle ^{00}\subseteq\theta\left(\mathcal{Z}^{00}\right)\subseteq\theta\left(W^{\prime}\right)$.
\end{proof}

\begin{prop}\label{P:CovPties2}
Let $\mathcal{U}$, $\mathcal{V}$ be locally definable groups with
identities $e_{\mathcal{U}}$ and $e_{\mathcal{V}}$, respectively,
and $\theta:\mathcal{U}\rightarrow\mathcal{V}$ a locally definable
covering homomorphism. Let $Y\subseteq\mathcal{V}$ be a definable simply connected set, and $Y^{\prime}\subseteq Y$
a connected definable set such that $e_{\mathcal{V}}\in Y^{\prime}$
and $Y^{\prime}Y^{\prime}\subseteq Y$, then there is a definable
set $W^{\prime}\subseteq \mathcal{U}$ such that $e_{\mathcal{U}}\in W^{\prime}$, $\theta\mid_{W^{\prime}}:W^{\prime}\rightarrow Y^{\prime}$
is a definable homeomorphism and a local homomorphism in both directions.

\[
\xymatrix@R=30pt{W^{\prime}\ar@{<-->}[d]_-{\theta\mid_{W^{\prime}}}\ar@{^{(}-->}[rr]^-{i} &  & \mathcal{U}\ar[d]^-{\theta}\\
Y^{\prime}\ar@{^{(}->}[r]^-{i} & Y^{\prime}Y^{\prime}\subseteq Y\ar@{^{(}->}[r]^-{i} & \mathcal{V}
}
\]
\end{prop}
\begin{proof}
By Proposition \ref{P:CovPties1},
there is a definable $W_{1}\subseteq\mathcal{U}$ open in $\theta^{-1}\left(Y\right)$
such that $e_{\mathcal{U}}\in W_{1}$ and $\theta\mid_{W_{1}^{0}}:W_{1}^{0}\rightarrow Y$
is a definable homeomorphism, where $W_{1}^{0}$ is the identity component
of $W_{1}$. Let $W^{\prime}=\theta\mid_{W_{1}^{0}}^{-1}\left(Y^{\prime}\right)$,
then $\theta\left(W^{\prime}W^{\prime}\right)=\theta\left(W^{\prime}\right)\theta\left(W^{\prime}\right)\subseteq Y^{\prime}Y^{\prime}\subseteq Y$,
then $W^{\prime}W^{\prime}\subseteq\theta^{-1}\left(Y\right)\subseteq W_{1}^{0}\ker\left(\theta\right)$.

In addition, $W_{1}^{0}k_{1}\cap W_{1}^{0}k_{2}=\emptyset$ if $k_{1}\neq k_{2}$
and $k_{1},k_{2}\in\ker\left(\theta\right)$; otherwise, if there are $y_{1},y_{2}\in W_{1}^{0}$ such that $y_{1}k_{1}=y_{2}k_{2}$, then
$\theta\left(y_{1}k_{1}\right)=\theta\left(y_{1}\right)=\theta\left(y_{2}k_{2}\right)=\theta\left(y_{2}\right)$,
but $\theta$ is injective in $W_{1}^{0}$, then $y_{1}=y_{2}$, so
$k_{1}=k_{2}$, which is a contradiction since $k_{1}\neq k_{2}$. Then $\theta\mid_{W_{2}}:W_{2}\rightarrow Y$
is a definable covering map of ld-spaces and $e_{\mathcal{U}}\in W_{2}$.

Therefore, from the connectedness of $W^{\prime}W^{\prime}$ and $W^{\prime}W^{\prime}\subseteq W_{1}^{0}\ker\left(\theta\right)$, we get $W^{\prime}W^{\prime}\subseteq W_{1}^{0}$. Thus, \cite[Remark 2.12]{BADefComp-I} implies that the homeomorphism $\theta\mid_{W^{\prime}}:W^{\prime}\rightarrow Y^{\prime}$ is
a local homomorphism in both directions.
\end{proof}

\section{Extension of a definable local homomorphism from a torsion free abelian locally definable group}\label{S:6}

\begin{thm}\label{T:Z_t-f}
Let $\mathcal{Z}$ be a connected abelian torsion free locally definable
group such that $\mathcal{Z}^{00}$ exists, and let $\mathcal{V}$
be an abelian locally definable group. Let $W\subseteq\mathcal{Z}$
be a definable set such that $\mathcal{Z}^{00}\subseteq W$. Assume
that $\theta:W\subseteq\mathcal{Z}\rightarrow\mathcal{V}$ is a definable local
homomorphism.

Then there exists a unique locally definable homomorphism $\overline{\theta}:\mathcal{Z}\rightarrow\mathcal{V}$ extending $\overline{\theta}$.

If in addition $\mathcal{V}$ is connected, $\mathcal{V}^{00}$
exists, $\mathcal{V}^{00}\subseteq\theta\left(W\right)$, $\theta$
is injective and $\theta^{-1}:\theta\left(W\right)\rightarrow W$
is a local homomorphism, then $\overline{\theta}:\mathcal{Z}\rightarrow\left\langle \theta\left(W\right)\right\rangle =\mathcal{V}$
is the o-minimal universal covering homomorphism of $\mathcal{V}$.

\[
\xymatrix{ &  & \mathcal{Z}\ar@{-->}@/_/[lld]_-{\overline{\theta}}\\
\mathcal{V} & \theta\left(W\right)\ar@{_{(}->}[l]_-{i} & W\ar[l]_-{\theta}\ar@{}[u]|*[@]{\subseteq}
}
\]
\end{thm}
\begin{proof}
By Proposition \ref{P:R1}, there is a definable symmetric $W_{1}\subseteq W$
such that $\mathcal{Z}^{00}\subseteq W_{1}\subseteq\prod_{4}W_{1}\subseteq W$, and $\theta\left(W_{1}\right)$ is generic in $\left\langle \theta\left(W_{1}\right)\right\rangle $.
Now, note that since $\mathcal{Z}^{00}$ exists, by \cite[Proposition 3.5]{PPant12}, $\mathcal{Z}$ is divisible, then the map $z\mapsto z^{k}:\mathcal{Z}\rightarrow\mathcal{Z}$
is a group isomorphism. By Proposition \ref{P:ImpProp}(iii), there
is $k\in\mathbb{N}$ such that the convex hull $\textrm{ch}\left(W_{1}^{\frac{1}{k}}\right)$
of $W_{1}^{\frac{1}{k}}$ is contained in $W_{1}$.

Let $y\in\mathcal{Z}=\left\langle W_{1}\right\rangle =\bigcup_{n\in\mathbb{N}}\prod_{n}W_{1}$,
so there are $n\in\mathbb{N}$ and $y_{1},\ldots,y_{n}\in W_{1}$
such that $y=y_{1}\cdots y_{n}$. Let \[
\overline{\theta}\left(y\right)=\left(\theta\left(y_{1}^{\frac{1}{k}}\right)\cdots\theta\left(y_{n}^{\frac{1}{k}}\right)\right)^{k}.
\]
\begin{claim}\label{C:TK-C4}
The map $\overline{\theta}:\mathcal{Z}\rightarrow\mathcal{V}$ defined
as above satisfies the following.
\begin{enumerate}[(i)]

\item $\overline{\theta}$ is a well defined map.

\item $\overline{\theta}$ is a locally definable homomorphism.

\item $\overline{\theta}$ is the unique extension of $\theta:W\rightarrow\mathcal{V}$ that is a locally definable homomorphism from $\mathcal{Z}$ to $\mathcal{V}$.

\item If, moreover, $\mathcal{V}$ is connected, $\mathcal{V}^{00}$ exists, $\mathcal{V}^{00}\subseteq\theta\left(W\right)$, $\theta$ is injective and $\theta^{-1}:\theta\left(W\right)\rightarrow W$ is a local homomorphism, then $\overline{\theta}\left(\mathcal{Z}\right)=\mathcal{V}$ and $\overline{\theta}$ is the o-minimal universal covering homomorphism of $\mathcal{V}$.
\end{enumerate}
\end{claim}
\begin{proof}
(i) As $ch\left(W_{1}^{\frac{1}{k}}\right)\subseteq W_{1}$, then
for every $i,j\in\mathbb{N}$ with $j\leq i$ and $i>0$ we have that: \[
\prod_{j}\left(W_{1}^{\frac{1}{k}}\right)^{\frac{1}{i}}\subseteq W_{1}.\]
And since $\theta$ is a locally homomorphism, then for every $y_{1},\ldots,y_{j}\in W_{1}^{\frac{1}{k}}$

\begin{equation}
\theta\left(\underbrace{y_{1}^{\frac{1}{i}}\cdots y_{j}^{\frac{1}{i}}}_{j\textrm{-times}}\right)=\theta\left(y_{1}^{\frac{1}{i}}\right)\cdots\theta\left(y_{j}^{\frac{1}{i}}\right).\label{Ec:Ast2}
\end{equation}

Now, we will see that $\overline{\theta}$ is well defined.

Let $y\in\mathcal{Z}=\left\langle W_{1}\right\rangle =\bigcup_{n\in\mathbb{N}}\prod_{n}W_{1}$,
and suppose that
\begin{equation}
y=y_{1}\cdots y_{n}=x_{1}\cdots x_{m}\label{Ec:Ast3}
\end{equation}
for some $y_{1},\ldots,y_{n},x_{1},\ldots,x_{m}\in W_{1}$, and $m,n\in\mathbb{N}$.
Additionally, assume, without loss of generality, that $m\leq n$.
\begin{align*}
\overline{\theta}\left(y_{1}\cdots y_{n}\right) & =  \left(\theta\left(y_{1}^{\frac{1}{k}}\right)\cdots\theta\left(y_{n}^{\frac{1}{k}}\right)\right)^{k}\\
 & =  \left(\theta\left(\underset{n\textrm{-times}}{\underbrace{y_{1}^{\frac{1}{kn}}\cdots y_{1}^{\frac{1}{kn}}}}\right)\cdots\theta\left(\underset{n\textrm{-times}}{\underbrace{y_{n}^{\frac{1}{kn}}\cdots y_{n}^{\frac{1}{kn}}}}\right)\right)^{k}\textrm{, by Eq. }(\ref{Ec:Ast2})\\
 & =  \left(\left(\theta\left(y_{1}^{\frac{1}{kn}}\right)\right)^{n}\cdots\left(\theta\left(y_{n}^{\frac{1}{kn}}\right)\right)^{n}\right)^{k}\textrm{, by Eq. }(\ref{Ec:Ast2})\\
 & =  \left(\left(\theta\left(y_{1}^{\frac{1}{kn}}\cdots y_{n}^{\frac{1}{kn}}\right)\right)^{n}\right)^{k}\textrm{, by Eq. }(\ref{Ec:Ast3})\\
 & =  \left(\left(\theta\left(x_{1}^{\frac{1}{kn}}\cdots x_{m}^{\frac{1}{kn}}\right)\right)^{n}\right)^{k}\textrm{, by Eq. }(\ref{Ec:Ast2})\\
 & =  \left(\left(\theta\left(x_{1}^{\frac{1}{kn}}\right)\right)^{n}\cdots\left(\theta\left(x_{m}^{\frac{1}{kn}}\right)\right)^{n}\right)^{k}\textrm{, by Eq. }(\ref{Ec:Ast2})\\
 & =  \left(\theta\left(x_{1}^{\frac{1}{k}}\right)\cdots\theta\left(x_{m}^{\frac{1}{k}}\right)\right)^{k}\\
 & =  \overline{\theta}\left(x_{1}\cdots x_{m}\right).
\end{align*}
Therefore, $\overline{\theta}$ is well defined.

(ii) Since $\overline{\theta}\mid_{\prod_{n}W_{1}}$ is a definable
map for every $n\in\mathbb{N}$, then the restriction of $\overline{\theta}$
to a definable subset of $\mathcal{Z}=\left\langle W_{1}\right\rangle $
is a definable map. And by definition of $\overline{\theta}$, $\overline{\theta}$
is clearly a group homomorphism.

(iii) First, we will see that $\theta\mid_{ch\left(W_{1}^{\frac{1}{k}}\right)}=\overline{\theta}\mid_{ch\left(W_{1}^{\frac{1}{k}}\right)}$.

By definition of the convex hull (Def. \ref{D:convexHull}) and the divisibility
of $\mathcal{Z}$, an element in $ch\left(W_{1}^{\frac{1}{k}}\right)\subseteq W_{1}$
is of the form $y_{1}^{\frac{1}{n}}\cdots y_{n}^{\frac{1}{n}}$ for
some $y_{1},\ldots,y_{n}\in W_{1}^{\frac{1}{k}}$ and $n\in\mathbb{N}$. Thus,

\begin{align*}
\overline{\theta}\left(y_{1}^{\frac{1}{n}}\cdots y_{n}^{\frac{1}{n}}\right) & =  \left(\theta\left(y_{1}^{\frac{1}{nk}}\cdots y_{n}^{\frac{1}{nk}}\right)\right)^{k}\textrm{, by Eq. }(\ref{Ec:Ast2})\\
 & =  \left(\theta\left(y_{1}^{\frac{1}{nk}}\right)\cdots\theta\left(y_{n}^{\frac{1}{nk}}\right)\right)^{k}\\
 & =  \left(\theta\left(y_{1}^{\frac{1}{nk}}\right)\right)^{k}\cdots\left(\theta\left(y_{n}^{\frac{1}{nk}}\right)\right)^{k}\textrm{, by Eq. }(\ref{Ec:Ast2})\\
 & = \theta\left(y_{1}^{\frac{1}{n}}\right)\cdots\theta\left(y_{n}^{\frac{1}{n}}\right)\textrm{, by Eq. }(\ref{Ec:Ast2})\\
 & =  \theta\left(y_{1}^{\frac{1}{n}}\cdots y_{n}^{\frac{1}{n}}\right).
\end{align*}

Then $\theta$ and $\overline{\theta}$ agree on $ch\left(W_{1}^{\frac{1}{k}}\right)$.

Now, we will verify uniqueness. Let $\beta:\mathcal{Z}\rightarrow\mathcal{V}$ be a locally definable homomorphism that is an extension of $\theta:W\rightarrow\mathcal{V}$, then in particular $\beta$ and $\overline{\theta}$ agree on $W_{1}^{\frac{1}{k}}$. Let $\mathcal{Z}^{\prime}=\left\{ z\in\mathcal{Z}:\beta\left(z\right)=\overline{\theta}\left(z\right)\right\} $. Then $G^{00}\subseteq W_{1}^{\frac{1}{k}}\subseteq\mathcal{Z}^{\prime}$ and $\mathcal{Z}^{\prime}$ is an open locally definable subgroup of $\mathcal{Z}$. By \cite[Lemma 2.12]{ED05}, $\mathcal{Z}^{\prime}$ is a compatible subset of $\mathcal{Z}$. But $\mathcal{Z}$ is connected, then $\mathcal{Z}=\mathcal{Z}^{\prime}$; i.e., $\beta=\overline{\theta}$.

\begin{comment}
Another proof of the fact that $\beta=\overline{\theta}$ can be obtained by proving that $\beta$ and $\overline{\theta}$ agree on $\prod_{n}W_{1}^{\frac{1}{k}}$ for every $n\in\mathbb{N}$, which
follows from the agreement of $\beta$ and $\overline{\theta}$ on $W_{1}^{\frac{1}{k}}$. And as $\left\langle W_{1}^{\frac{1}{k}}\right\rangle =\mathcal{Z}$, so $\beta=\overline{\theta}$.
\end{comment}

\begin{comment}
\begin{eqnarray*}
\beta\left(y_{1}\cdots y_{n}\right) & = & \beta\left(y_{1}\right)\cdots\beta\left(y_{n}\right)\textrm{, by }\beta\mid_{W_{1}^{\frac{1}{k}}}=\theta\mid_{W_{1}^{\frac{1}{k}}}\\
 & = & \theta\left(y_{1}\right)\cdots\theta\left(y_{n}\right)\textrm{, by }\overline{\theta}\mid_{W_{1}^{\frac{1}{k}}}=\theta\mid_{W_{1}^{\frac{1}{k}}}\\
 & = & \overline{\theta}\left(y_{1}\right)\cdots\overline{\theta}\left(y_{n}\right)\textrm{, since }\overline{\theta}\textrm{ is a homomorphism}\\
 & \text{=} & \overline{\theta}\left(y_{1}\cdots y_{n}\right).
\end{eqnarray*}
\end{comment}

Then, $\beta=\overline{\theta}$. And hence, $\overline{\theta}\mid_{W}=\theta$; i.e., $\overline{\theta}$ is also an extension of $\theta:W\rightarrow\mathcal{V}$.

(iv) First, note that $\overline{\theta}\left(\mathcal{Z}\right)=\bigcup_{n\in\mathbb{N}}\prod_{n}\left(\theta\left(W_{1}^{\frac{1}{k}}\right)\right)^{k}=\left\langle \left(\theta\left(W_{1}^{\frac{1}{k}}\right)\right)^{k}\right\rangle $.
Now, by Proposition \ref{P:ImpProp}(ii), there is $l\in\mathbb{N}$
such that $W\subseteq W_{1}^{l}$, then $\theta^{-1}\left(\mathcal{V}^{00}\right)\subseteq W\subseteq W_{1}^{l}$.
By the hypothesis on $\theta^{-1}$,
$\theta^{-1}\left(\mathcal{V}^{00}\right)$ is a type-definable subgroup
of $\mathcal{Z}$; moreover, by \cite[Proposition 3.5]{PPant12}, $\mathcal{V}^{00}$ is divisible, so $\theta^{-1}\left(\mathcal{V}^{00}\right)$
is an abelian torsion free divisible subgroup of $\mathcal{Z}$. Henceforth,
$\theta^{-1}\left(\mathcal{V}^{00}\right)\subseteq W_{1}^{l}$ implies
that $\theta^{-1}\left(\mathcal{V}^{00}\right)=\theta^{-1}\left(\mathcal{V}^{00}\right)^{\frac{1}{kl}}\subseteq W_{1}^{\frac{1}{k}}$,
thus $\mathcal{V}^{00}\subseteq\left(\theta\left(W_{1}^{\frac{1}{k}}\right)\right)^{k}$,
it follows that $\mathcal{V}=\left\langle \left(\theta\left(W_{1}^{\frac{1}{k}}\right)\right)^{k}\right\rangle =\left\langle \theta\left(W\right)\right\rangle $. Then, $\overline{\theta}:\mathcal{Z}\rightarrow\mathcal{V}$ is surjective.

On the other hand, notice that $\dim\left(\ker\left(\overline{\theta}\right)\right)=0$
if and only if $\dim\left(\mathcal{Z}\right)=\dim\left(\overline{\theta}\left(\mathcal{Z}\right)\right)=\dim\left(\left\langle \theta\left(W_{1}\right)\right\rangle \right)$.
Since $W_{1}$ and $\theta\left(W_{1}\right)$ are generic in $\mathcal{Z}$
and $\left\langle \theta\left(W_{1}\right)\right\rangle $, respectively,
then $\dim\left(\mathcal{Z}\right)=\dim\left(W_{1}\right)$ and $\dim\left(\left\langle \theta\left(W_{1}\right)\right\rangle \right)=\dim\left(\theta\left(W_{1}\right)\right)$;
finally, $\dim\left(\mathcal{Z}\right)=\dim\left(\left\langle \theta\left(W_{1}\right)\right\rangle \right)$
follows from the injectivity of $\theta$, so $\overline{\theta}:\mathcal{Z}\rightarrow\mathcal{V}$
is a locally definable covering homomorphism by \cite[Theorem 3.6]{ED05}.

Since $\mathcal{Z}$ is abelian, connected, and $\mathcal{Z}^{00}$ exists, then $\mathcal{Z}$ covers an abelian definable group (\cite[Thm. 3.9]{PPant12}). Thus Claim \ref{C:tf-is-sc} yields $\mathcal{Z}$ is simply connected. Therefore, by \cite[Remark 3.8]{EdPanPre2013}, $\overline{\theta}:\mathcal{Z}\rightarrow\mathcal{V}$ is the o-minimal universal covering homomorphism of $\mathcal{V}$.
\end{proof}
This concludes the proof of Theorem \ref{T:Z_t-f}.
\end{proof}

\section{Universal covers of locally homomorphic abelian locally definable groups}\label{S:7}

\begin{thm}\label{T:UnivCovAbLD}
Let $\mathcal{U}$, $\mathcal{V}$ be abelian connected locally definable
groups such that $\mathcal{U}^{00}$ exists and $\mathcal{U}^{00}$ is an intersection of $\omega$-many simply
connected definable subsets of $\mathcal{U}$. Let $X\subseteq\mathcal{U}$ be a definable set with $\mathcal{U}^{00}\subseteq X$, and $\phi:X\subseteq\mathcal{U}\rightarrow \phi\left(X\right)\subseteq\mathcal{V}$
a definable homeomorphism and a local homomorphism. Assume that $\pi:\widetilde{\mathcal{V}}\rightarrow\mathcal{V}$ is the o-minimal universal covering homomorphism of $\mathcal{V}$.

Then,
\begin{enumerate}[(i)]

\item there are a connected locally definable subgroup $\mathcal{W}$ of $\widetilde{\mathcal{V}}$ and a locally definable homomorphism $\overline{\theta}:\mathcal{W}\rightarrow\mathcal{U}$ that is the o-minimal universal covering homomorphism of $\mathcal{U}$, and
\item there is a connected symmetric definable $X^{\prime}\subseteq X$ with $\mathcal{U}^{00}\subseteq X^{\prime}$ such that $\pi\mid_{\mathcal{W}}:\mathcal{W}\rightarrow\left\langle \phi\left(X^{\prime}\right)\right\rangle \leq\mathcal{V}$ is the o-minimal universal covering homomorphism of $\left\langle \phi\left(X^{\prime}\right)\right\rangle $.
\end{enumerate}

If in addition $X$ is simply connected, then $\mathcal{W}$ is a subgroup of the o-minimal universal covering group $\widetilde{\left\langle \phi\left(X\right)\right\rangle }$
of $\left\langle \phi\left(X\right)\right\rangle $.

\[
\xymatrix{ &  & \mathcal{W}\ar@{-->}@/_/[lld]_-{\overline{\theta}}\ar@{^{(}-->}[r]^-{i} & \widetilde{\mathcal{V}}\ar[d]^-{\pi}\\
\mathcal{U} & X\ar[r]^-{\phi}\ar@{_{(}->}[l]_-{i} & \phi\left(X\right)\ar@{^{(}->}[r]^-{i} & \mathcal{V}
}
\]

\end{thm}

\begin{proof}

Since $\mathcal{U}^{00}$ is an intersection of $\omega$-many simply
connected definable subsets of $\mathcal{U}$, then $\mathcal{U}^{00}\subseteq X$ and saturation imply that there are simply connected definable sets $X_{1}$ and $X_{2}$ such
that $\mathcal{U}^{00}\subseteq X_{2}\subseteq X_{2}X_{2}\subseteq X_{1}\subseteq X$. Thus, by \cite[Remark 2.12]{BADefComp-I}, $\phi\mid_{X_{2}}$ is a local homomorphism in
both directions.

Moreover, the connected definable set $Y_{2}=\phi\left(X_{2}\right)$
is such that $Y_{2}Y_{2}\subseteq Y_{1}$ where $Y_{1}=\phi\left(X_{1}\right)$ and the identity of $\mathcal{V}$ belongs to $Y_{2}$. So Proposition \ref{P:CovPties2} yields the existence of a connected definable $W_{2}\subseteq\widetilde{\mathcal{V}}$
such that the identity of $\widetilde{\mathcal{V}}$ is in $W_2$ and $\pi\mid_{W_{2}}:W_{2}\rightarrow Y_{2}$ is a definable
homeomorphism and a local homomorphism in both directions, hence so is $\theta=\phi\mid_{X_{2}}^{-1}\circ\pi\mid_{W_{2}}$.

By saturation, $\mathcal{U}^{00}\subseteq X_{2}$, and Proposition \ref{P:R1}, then there is a connected
symmetric definable $X^{\prime}\subseteq X_{2}$ such that (i) $\mathcal{U}^{00}\subseteq X^{\prime}\subseteq\prod_{4}X^{\prime}\subseteq X_{2}$,
(ii) $W_{3}=\theta^{-1}\left(X^{\prime}\right)$ is generic in $\mathcal{W}=\left\langle W_{3}\right\rangle \leq\widetilde{\mathcal{V}}$,
$\mathcal{W}^{00}$ exists, and $\mathcal{W}^{00}\subseteq W_{3}$,
and (iii) $Y_{3}=\phi\left(X^{\prime}\right)$ is generic in $\left\langle Y_{3}\right\rangle \leq\mathcal{V}$,
$\left\langle Y_{3}\right\rangle ^{00}$ exists, and $\left\langle Y_{3}\right\rangle ^{00}\subseteq Y_{3}$. Note that $\mathcal{W}$ is connected by Proposition \ref{P:DefGenIsConn}.

By Theorem \ref{T:Z_t-f}, there are locally definable homomorphisms
$\overline{\theta}:\mathcal{W}\rightarrow\mathcal{U}$ and $\overline{\pi}:\mathcal{W}\rightarrow\left\langle Y_{3}\right\rangle \leq\mathcal{V}$
that are the o-minimal universal covering homomorphisms of $\mathcal{U}$ and $\left\langle Y_{3}\right\rangle$, respectively. Moreover, $\overline{\theta}:\mathcal{W}\rightarrow\mathcal{U}$ and $\overline{\pi}:\mathcal{W}\rightarrow\left\langle Y_{3}\right\rangle \leq\mathcal{V}$ are the unique extensions of $\theta\mid_{W_{3}}:W_{3}\rightarrow\mathcal{U}$ and $\pi\mid_{W_{3}}:W_{3}\rightarrow Y_{3}\subseteq\mathcal{V}$,
respectively, that are locally definable homomorphisms. Since $\overline{\pi}\mid_{W_{3}}=\pi\mid_{W_{3}}$,
then $\overline{\pi}=\pi\mid_{\mathcal{W}}$. If in addition
$X$ is simply connected, then  the o-minimal universal covering group $\widetilde{\left\langle \phi\left(X\right)\right\rangle }$
of $\left\langle \phi\left(X\right)\right\rangle$ exists, and $\mathcal{W}$
is a closed subgroup of $\widetilde{\left\langle \phi\left(X\right)\right\rangle}$.

\end{proof}

\subsection{The o-minimal universal covering group of an abelian connected definably compact semialgebraic group}

Applying the main results obtained so far, we present below one of the key results of this work.

\begin{thm}\label{T:Goal1}
Let $G$ be an abelian connected definably compact group definable
in a sufficiently saturated real closed field $R$. Then there are a connected $R$-algebraic group $H$,
an open connected locally definable subgroup $\mathcal{W}$ of the
 o-minimal universal covering group $\widetilde{H\left(R\right)^{0}}$ of $H\left(R\right)^{0}$, and
a locally definable homomorphism $\overline{\theta}:\mathcal{W}\subseteq\widetilde{H\left(R\right)^{0}}\rightarrow G$
that is the o-minimal universal covering homomorphism of $G$.
\end{thm}
\begin{proof}
By \cite[Theorem 5.1]{BADefComp-I}, there are a connected $R$-algebraic group $H$ such that $\dim\left(G\right)=\dim\left(H\left(R\right)\right)$, a definable set $X\subseteq G$ such that $G^{00}\subseteq X$, and a definable homeomorphism
$\phi:X\subseteq G\rightarrow\phi\left(X\right)\subseteq H\left(R\right)$
that is also a local homomorphism.

By \cite[Thm. 2.2]{Berar09}, $G^{00}$ is an intersection of $\omega$-many simply
connected definable subsets of $G$. Thus, by Theorem \ref{T:UnivCovAbLD}, there are a connected locally definable
subgroup $\mathcal{W}\leq\widetilde{H\left(R\right)^{0}}$ and a locally definable homomorphism
$\overline{\theta}:\mathcal{W}\subseteq\widetilde{H\left(R\right)^{0}}\rightarrow G$
that is the o-minimal universal covering homomorphism of $G$, and
since $\dim\left(G\right)=\dim\left(\mathcal{W}\right)$, $\mathcal{W}$
is also open in $\widetilde{H\left(R\right)^{0}}$.
\end{proof}

%%%
\section*{Acknowledgements}
I would like to express my gratitude to the Universidad de los Andes,
Colombia and the University of Haifa, Israel for supporting and funding my research as well as for their stimulating hospitality. I would also like to thank warmly to my advisors: Alf Onshuus and Kobi
Peterzil for their support, generous ideas, and kindness during this work.

I want to express my gratitude to Anand Pillay for suggesting to Kobi Peterzil the problem discussed in this work: the study of semialgebraic groups over a real closed field. Also, thanks to the Israel-US Binational Science Foundation for their support.

The main results of this paper have been presented on the winter of 2016 at the Logic Seminar of the Institut Camille Jordan, Université Claude Bernard - Lyon 1 (Lyon) and at the Oberseminar Modelltheorie of the Universit{\"a}t Konstanz (Konstanz).

%%%%%%%%%%%%%%%%
% bibliography
%%%%%%%%%%%%%%%
\nocite{ErraEdCov,BaEdErrata,PetStein99,HP2,TentZie,Hodges}
\bibliographystyle{plain}
\bibliography{IP}

\end{document}